\documentclass[12pt, reqno]{amsart}

\usepackage[a4paper,top=3cm,bottom=2.5cm,left=2.75cm,right=2.75cm,marginparwidth=1.75cm]{geometry}

\usepackage{amsmath}
\usepackage{amsthm}
\usepackage{amsfonts}
\usepackage{amssymb}
\usepackage{mathtools}
\usepackage[mathscr]{euscript}
\usepackage{bm}
\usepackage[table]{xcolor}
\usepackage[all,cmtip]{xy}
\usepackage{enumitem}
\usepackage{tikz}  
\usepackage{tikz-cd}
\usepackage{graphicx}
\usepackage{scalerel}
\usepackage{microtype}

\usepackage{verbatim}
\usepackage{libertine}

\usepackage[pdftex,
              pdfauthor={Hannah Dell, Augustinas Jacovskis, Franco Rota},
              pdftitle={},
              pdfsubject={},
              pdfkeywords={},
              colorlinks=true, 
              citecolor=citation, 
              urlcolor=citation, 
              linkcolor=reference]{hyperref}

\makeatletter 
\@mparswitchfalse%
\makeatother
\normalmarginpar 
\usepackage[color=todogray,textwidth=20mm,textsize=tiny]{todonotes}

\usepackage{comment}
\usepackage{array}
\usepackage{tikz-cd}

\definecolor{citation}{rgb}{0,.40,.80}
\definecolor{reference}{rgb}{.80,0,.40}
\definecolor{todogray}{HTML}{E4E4E4}

\numberwithin{equation}{section}

\theoremstyle{plain}
\newtheorem{theorem}{Theorem}[section]
\newtheorem{lemma}[theorem]{Lemma}
\newtheorem{proposition}[theorem]{Proposition}
\newtheorem{corollary}[theorem]{Corollary}

\newtheorem{question}[theorem]{Question}

\theoremstyle{definition}
\newtheorem{definition}[theorem]{Definition}
\newtheorem{example}[theorem]{Example}
\newtheorem{remark}[theorem]{Remark}

\newtheorem*{remark*}{Remark}

\newcommand{\Db}{\mathrm{D^b}}
\newcommand{\Dperf}{\mathrm{D}_{\mathrm{perf}}}

\newcommand{\sB}{\mathcal{B}}

\DeclareMathOperator{\Pic}{Pic}
\DeclareMathOperator{\pt}{pt}

\DeclareMathOperator{\Cone}{Cone}

\DeclareMathOperator{\Spec}{Spec}

\DeclareMathOperator{\Hom}{Hom}

\DeclareMathOperator{\Ext}{Ext}
\DeclareMathOperator{\Aut}{Aut}

\DeclareMathOperator{\HH}{HH}

\newcommand{\id}{\mathrm{id}}

\newcommand{\Cat}{\mathrm{Cat}}

\renewcommand{\top}{\mathrm{top}}

\DeclareMathOperator{\HP}{HP}

\newcommand{\prim}{\mathrm{prim}}

\newcommand{\GLcov}{\widetilde{\mathrm{GL}}}

\newcommand{\hooklongrightarrow}{\lhook\joinrel\longrightarrow}

\newcommand{\ch}{\mathrm{ch}}
\newcommand{\rk}{\mathrm{rk}}
\newcommand{\td}{\mathrm{td}}

\docsvlist{A,B,C,D,E,F,G,H,I,J,K,L,M,N,O,P,Q,R,S,T,U,V,W,X,Y,Z} 

\docsvlist{A,B,C,D,E,F,G,H,I,J,K,L,M,N,O,P,Q,R,S,T,U,V,W,X,Y,Z} 

\newcommand{\rK}{\mathrm{K}}

\newcommand{\bC}{\mathbf{C}}

\newcommand{\bZ}{\mathbf{Z}}
\newcommand{\bP}{\mathbf{P}}
\newcommand{\bQ}{\mathbf{Q}}
\newcommand{\bR}{\mathbf{R}}
\newcommand{\bL}{\mathbf{L}}

\newcommand{\sfR}{\mathsf{R}}

\newcommand{\sfS}{\mathsf{S}}

\usepackage[numbers]{natbib}
\begin{document}

\title[Cyclic covers: Hodge theory and categorical Torelli]{Cyclic covers: Hodge theory and categorical Torelli theorems}

\author[H. Dell]{Hannah Dell}
\address{HD: School of Mathematics and Maxwell Institute, University of Edinburgh, James Clerk Max\-well Building, Peter Guthrie Tait Road, Edinburgh, EH9 3FD, United Kingdom}
\email{h.dell@sms.ed.ac.uk}

\author[A. Jacovskis]{Augustinas Jacovskis}
\address{AJ: Department of Mathematics, Maison du Nombre, 6 avenue de la Fonte, L-4364 Esch-sur-Alzette, Luxembourg}
\email{augustinas.jacovskis@uni.lu}

\author[F. Rota]{Franco Rota}
\address{FR: School of Mathematics and Statistics, University of Glasgow, Glasgow G12 8QQ, United Kingdom} 
\email{franco.rota@glasgow.ac.uk}

\thanks{ HD and AJ were supported by ERC Consolidator Grant WallCrossAG, no. 819864. AJ was also supported by the Luxembourg National Research Fund (FNR–17113194). FR is supported by EPSRC grant EP/R034826/1 and acknowledges support from the Deutsche Forschungsgemeinschaft (DFG, German Research Foundation) under Germany's Excellence Strategy -- EXC-2047/1 -- 390685813.}


\maketitle

\begin{abstract}
Let $Y$ admit a rectangular Lefschetz decomposition of its derived category, and consider a cyclic cover $X\to Y$ ramified over a divisor $Z$. In a setting not considered by Kuznetsov and Perry \cite{kuznetsov2017derived}, we define a subcategory $\sA_Z$ of the equivariant derived category of $X$ which contains, rather than is contained in, $\Db(Z)$. We then show that the equivariant category of the Kuznetsov component of $X$ is decomposed into copies of $\sA_Z$.

As an application,
we relate $\sA_Z$ with the cohomology of $Z$ under some numerical assumptions. In particular, we obtain categorical Torelli theorems for the lowest degree prime Fano threefolds of index 1 and 2.

\end{abstract}

\section{Introduction}

In \cite{kuznetsov2009derived}, Kuznetsov systematically studies derived categories of Fano threefolds and subcategories -- known as \textit{Kuznetsov components} -- arising as orthogonal complements to exceptional collections. These categories are linked to questions of rationality \cite{kuznetsov2010cubic}, stability conditions \cite{bayer2017stability} and to hyperk\"ahler geometry through moduli spaces \cite{LPZ18twistedcubics}, as well as questions in birational geometry and Hodge theory \cite{perry2022integral}.

In this paper, we start from a variety $Y$ admitting a sheaf $\sO_Y(1)$ and a rectangular Lefschetz decomposition of $\Db(Y)$, i.e. an admissible subcategory $\sB$ and a decomposition
\begin{equation*}
    \Db(Y) = \langle \sB, \dots, \sB(m-1) \rangle,
\end{equation*}
with $\sB(i)$ the image of $\sB$ under the equivalence $-\otimes \sO_Y(1)$. Examples of varieties with rectangular decompositions are projective space, Grassmannians, some homogeneous spaces (see \cite[Section 4.1]{kuznetsov2019calabi}). Another important case, also covered in this article, is when $Y$ is a Deligne--Mumford stack, for example a weighted projective space. Then, we consider an $n$-fold cover $X$ of $Y$, ramified along a divisor $Z \in |\sO_Y(nd)|$. The derived category $\Db(X)$ has a semiorthogonal decomposition \[ \Db(X) = \langle \sA_X, \sB_X, \dots, \sB_X(M-1) \rangle \] where $\sB_X$ is the pullback of $\sB$ under the covering map, and $M \coloneqq m-(n-1)d$.  The action of $\mu_n$ on $X$ by the covering involution restricts to one on $\sA_X$, giving rise to the Kuznetsov component $\sA_X^{\mu_n} \subset \Db(X)^{\mu_n}$. This is the same setting of \cite{kuznetsov2017derived}, with the crucial difference that we work in the range $0<M<d$, rather than $d\leq M$. The former corresponds to allowing more positivity for $Z$ (see Section \ref{relation to other works}). In \cite{kuznetsov2017derived}, $\Db(Z)$ also admits a Kuznetsov component $\sA_Z$, and $\sA_X^{\mu_n}$ decomposes into copies of $\sA_Z$. We define a subcategory $\sA_Z \subset \Db(X)^{\mu_n}$ which \textit{contains}, rather than is contained in, $\Db(Z)$ (see \eqref{eq:A_Z_definition}), and in our first main result we exhibit a semiorthogonal decomposition for $\sA_{X}^{\mu_n}$ with components $\sA_Z$.

\begin{theorem}[= Theorem {\ref{thm:main}}]\label{thm:1.1}
    Assume that we are in the situation above, and that $0<M<d$. Then there are fully faithful functors $\Phi_j\colon \sA_Z \to \sA_X^{\mu_n}$, with $j=0,\ldots, n-2$, and a semiorthogonal decomposition:
        \begin{equation*}
        \sA_X^{\mu_n} = \langle \Phi_0(\sA_Z) , \dots, \Phi_{n-2}(\sA_Z) \rangle \subset \Db(X)^{\mu_n}.
    \end{equation*}
\end{theorem}

In the next part of the paper, we study the topological $\rK $-theory of $\sA_Z$, and its associated Hodge structure. For this, we use Blanc's topological $\rK $-theory for categories \cite{blanc2016topological} and the non-commutative Hodge theory methods of \cite{perry2022integral} and \cite{DHLncHDR}, slightly extending their results to admissible subcategories of global quotient DM stacks (Proposition \ref{prop:HS_C_in_Stack}). 

We apply this to the case of very general three-dimensional double covers of weighted projective space. Here, we can make a direct comparison between $\rK _0^\top(\sA_Z)$ and $H^*(Z,\bZ)$, and show that primitive cohomology of $Z$ coincides with the part $\rK_0(\sA_X^{\mu_2})^\perp$ of the topological $\rK $-theory of $\sA_X^{\mu_2}$ which is orthogonal to the algebraic classes. 

\begin{proposition}[{= Proposition \ref{prop:hs_pres_other}}] \label{intro:KalgPerp_Hprim}
Suppose that $X$ is a very general prime double cover of weighted projective three-dimensional space, with $0<M\leq d$. Then $\rK _0(\sA_X^{\mu_2})^\perp$ is equipped with a Hodge structure such that 
the Chern character map induces a Hodge isometry \[\rK_0(\sA_X^{\mu_2})^\perp \simeq H_\prim^2(Z, \bZ).\] 
\end{proposition}

A landmark theorem by Bondal--Orlov \cite{bondal2001reconstruction} states that the derived category of a Fano variety $X$ determines $X$ up to isomorphism. This justifies the question of whether $\sA_X$ suffices to determine $X$. When this is the case we say that a \textit{categorical Torelli theorem} holds for $X$. Motivated by these questions, we apply Proposition \ref{intro:KalgPerp_Hprim} to reconstruct primitive cohomology from categorical equivalences:

\begin{theorem}[= Theorem {\ref{thm:main_hodge}}]\label{thm:1.3}
    Let $X$ and $X'$ be prime double covers of a weighted projective three-dimensional space, with $X$ very general, $d=d'$ and $0<M\leq d$. A Fourier--Mukai type equivalence $\Phi^{\mu_2} \colon \sA_X^{\mu_2} \to \sA_{X'}^{\mu_2}$ induces a Hodge isometry
    \begin{equation*}H^{2}_\prim(Z, \bZ) \simeq H^{2}_\prim(Z', \bZ).
    \end{equation*}
\end{theorem}

Combining Theorem \ref{thm:1.3} with Hodge-theoretic Torelli theorems \cite{donagi1983generic, saito1986weak}, we prove new categorical Torelli theorems for two families of prime Fano threefolds, namely for threefolds $X_2$ of index $1$ and genus $2$, and for threefolds $Y_1$ of index $2$ and degree 1. In the case of $Y_1$, known as the \textit{Veronese double cone}, we impose a natural condition on the equivalence (we comment on this in Section \ref{sec:intro_related}). This was the only remaining open case for index 2.

\begin{theorem}[{= Theorem \ref{thm:X_2}}] \leavevmode
Let $X, X'$ be index $1$, genus $2$ prime Fano threefolds, with $X$ very general. Suppose there is a Fourier--Mukai type equivalence of Kuznetsov components $\sA_{X} \simeq \sA_{X'}$. Then $X \simeq X'$.
\end{theorem}

\begin{theorem}[= Theorem \ref{thm:Y_1}]\label{thm:1.5}
    Let $X,X'$ be Veronese double cones with $X$ very general. Suppose that we have an equivalence $\Phi\colon\sA_X \xrightarrow{\sim} \sA_{X'}$ of Fourier--Mukai type which commutes with the covering involution.
    Then $X \simeq X'$.
\end{theorem}

Our methods also allow us to reprove categorical Torelli for quartic double solids, i.e.\ prime Fano threefolds of index $2$ and degree $2$ (see Theorem \ref{thm:Y_2}). 

\subsection{Related works}\label{sec:intro_related}
If we compare Theorem \ref{thm:1.1} and \cite[Theorem 1.1]{kuznetsov2017derived} in the case where the base $Y$ is weighted projective space, we observe that the cases $0<M<d$, $M=d$, and $M>d$ correspond to the branch divisor $Z$ being canonically polarized, $K$-trivial, or Fano (see Section \ref{relation to other works}).
This is parallel to the trichotomy evidenced by Orlov \cite[Theorem 3.11]{Orlov09_DerSing} and suggests a relation between $\sA_Z$ and categories of matrix factorizations in the hypersurface setting. 

In the case of hypersurfaces, Hirano and Ouchi proved a result analogous to Theorem \ref{thm:1.1} in the language of matrix factorizations \cite[Theorem 5.6]{hirano2023derived}. 
From a similar perspective, the remarkable paper \cite{BFK} investigates Hodge theoretic applications of the language of matrix factorizations. 
The same philosophy was recently applied to prove categorical Torelli theorems for hypersurfaces in $\bP^n$ \cite[Theorem 1.3]{lin2023serre}.
A modification of this also applies to $X_2$, regarded as a sextic hypersurface in $\bP(1,1,1,1,3)$ \cite[Theorem 6.1]{lin2023serre}. The upcoming work \cite{LPS23} treats this case with independent methods.

Categorical Torelli theorems have been studied for many varieties. We direct the reader to \cite{PSSurvey} for a survey. One possible approach to proving such theorems is to recover the Fano threefold from Bridgeland moduli spaces, as in \cite{bernardara2012categorical, APR19, pertusi2022some, bayer2022desingularization, feyzbakhsh2023new} (for prime Fano threefolds of index $2$ and degree $\geq 2$) and \cite{jacovskis2021categorical, jacovskis2022brill} (for prime Fano theefolds of index $1$ and genus $\geq 6$).
These methods use the fact that there is a unique $\GLcov_2^+(\bR)$-orbit of Serre invariant stability conditions to show that any equivalence preserves  moduli spaces of stable objects. However, for the Veronese double cone and  lower genus index $1$ Fano threefolds, it is not known (or even false, e.g. genus 4 \cite[Corollary 1.9]{kuznetsov2021serre}) that such a unique orbit exists. Methods relating Hochschild cohomology and Hodge theory are developed in \cite{huybrechts2016hochschild} and \cite{pirozhkov2022categorical} to prove categorical Torelli theorems for broad classes of hypersurfaces.

In \cite{jacovskis2021categorical}, a categorical Torelli theorem for special Gushel--Mukai threefolds $X$ is shown by using the fact that the equivariant Kuznetsov component of $X$ is equivalent (by \cite{kuznetsov2017derived}) to the derived category of the K3 surface that $X$ is ramified in. Because derived equivalent K3 surfaces of this type have only one trivial Fourier--Mukai partner, the ramification divisors are isomorphic, hence the special Gushel--Mukai threefolds are isomorphic.

The additional assumption in Theorem \ref{thm:1.5} is equivalent to requiring that the equivalence $\Phi$ commutes with the \emph{rotation functor} $\sfR(-) = \bL_{\sO}(- \otimes \sO_X(1))$, also called the \textit{degree-shift} functor in \cite{huybrechts2016hochschild} (see Remark \ref{rmk:torelliY1}). 
This assumption is natural: for example, it is necessary for a categorical Torelli theorem for cubic fourfolds \cite[Corollary 1.2]{huybrechts2016hochschild}. In fact, some special cubic fourfolds have non-trivial \emph{Fourier--Mukai partners}, i.e. non-isomorphic cubic fourfolds sharing equivalent Kuznetsov components. In \cite[Section 3.2]{Huy17_K3CatCubic}, Huybrechts showed that the number of (isomorphism classes of) Fourier--Mukai partners is finite, and that a very general cubic fourfold has no non-trivial Fourier--Mukai partners. The articles \cite{pertusi2021fourier,FL23} study counts of Fourier--Mukai partners of a cubic fourfold. Inspired by these results, and by a question of Huybrechts \cite[Question 3.25]{MSK3Lectures}, we formulate the following question.
    
    \begin{question} \leavevmode
    \begin{enumerate}
        \item Do Veronese double cones have non-trivial Fourier--Mukai partners?
        \item Do \emph{very general} Veronese double cones have non-trivial Fourier--Mukai partners?
        \item If (1) has a positive answer, then how many Fourier--Mukai partners are there up to isomorphism, and are they birational to one another?
    \end{enumerate}
    \end{question}

\subsection*{Structure of the paper}

After recalling some preliminaries and fixing notations in Section \ref{sec:preliminaries}, we dedicate Section \ref{sec:covers} to the study of semiorthogonal decompositions for cyclic covers and the proof of Theorem \ref{thm:1.1}. Section \ref{ssec:equivar_equivalences} is dedicated to topological K-theory and non-commutative Hodge theory. We apply these tools to the threefold setting in Section \ref{sec:doublesolids}. The last section (Section \ref{sec:torelli}) contains the categorical Torelli theorems.

\subsection*{Notation and conventions}

We work over the field of complex numbers $\bC$. Whenever we work with an arbitrary triangulated category $\sC$, it will be assumed to be proper and $\bC$-linear. For a projective variety (or a proper DM stack) $X$, we write $\mathrm{D}(X)$ (resp. $\Db(X)$, $\Dperf(X)$) for the derived category of (bounded, perfect) complexes of coherent sheaves on $X$.

\subsection*{Acknowledgements}

We thank Arend Bayer, Pieter Belmans, Wahei Hara, Alex Perry, and Sebastian Schlegel Mejia for helpful discussions. We thank Laura Pertusi, Paolo Stellari, and Shizhuo Zhang for informing us about their work on similar topics.

\section{Preliminaries on derived categories}\label{sec:preliminaries}

For background on triangulated categories and derived categories, we recommend \cite{huybrechts2006fourier}. Throughout this section, $\sC$ is a triangulated category. For objects $E,F \in \sC$, define 
\[ \Hom^\bullet_{\sC}(E,F) \coloneqq \bigoplus_{t \in \bZ} \Ext^t_{\sC}(E,F)[-t] . \]
We will omit the Hom subscripts when the category we are working in is clear from context.

\begin{definition}
    An object $E \in \sC$ is called \emph{exceptional} if $\Hom^\bullet(E, E) = \bC$.
\end{definition}

\begin{definition}
    We say $\sC = \langle \sA_1, \dots, \sA_n \rangle$ is a \emph{semiorthogonal decomposition} of $\sC$ if 
    \begin{enumerate}
        \item $\Hom^\bullet(F, G) = 0$ for all $F \in \sA_i, G \in \sA_j$ if $i > j$;
        \item for any $F \in \sC$, there exists a sequence of morphisms 
        \[ 0 = F_n \to F_{n-1} \to \cdots \to F_1 \to F_0 = F \]
        such that $\Cone(F_i \to F_{i-1}) \in \sA_i$.
    \end{enumerate}
\end{definition}

Let $i \colon \sA \hookrightarrow \sC$ be a full  triangulated subcategory. If the inclusion $i$ has a left adjoint $i^*$, then $\sA$ is called \emph{left admissible}. If $i$ has a right adjoint $i^!$, then $\sA$ is called \emph{right admissible}. If both adjoints exist, then $\sA$ is called \emph{admissible}.

Define the \emph{right orthogonal} $\sA^\bot$ to $\sA$ to be the subcategory
\begin{equation*}
    \sA^\bot \coloneqq \{ F \in \sC \mid \Hom^\bullet(G, F)=0 \, \text{for all} \, G \in \sA \}.
\end{equation*}
Similarly, define the \emph{left orthogonal} ${}^\bot \sA$ to $\sA$ to be the subcategory
\begin{equation*}
    {}^\bot \sA \coloneqq \{ F \in \sC \mid \Hom^\bullet(F, G)=0 \, \text{for all} \, G \in \sA \}.
\end{equation*}

Let $i \colon \sA \hookrightarrow \sC$ be an admissible subcategory. Let $i^*$ and $i^!$ be the left and right adjoint, respectively. For any $F \in \sC$, we define the \emph{left mutation functor} $\bL_{\sA} (F)$ by the triangle
\begin{equation} \label{eq:left_mutation_def}
    i i^! (F) \to F \to \bL_{\sA}(F) .
\end{equation}
Similarly, we define the \emph{right mutation functor} $\bR_{\sA}$(F) by the triangle
\begin{equation} \label{eq:right_mutation_def}
    \bR_{\sA}(F) \to F \to ii^*(F) .
\end{equation}
In particular, when $\sA$ is an exceptional object $E$, these triangles become
\begin{equation}
    \Hom^\bullet(E, F) \otimes E \to F \to \bL_{E}(F),
\end{equation}
and 
\begin{equation}
    \bR_E(F) \to F \to \Hom^\bullet(F, E)^\vee \otimes E,
\end{equation}
respectively.

\begin{proposition} \label{prop:serre_sod}
Let $\sC = \langle \sC_1 , \sC_2 \rangle$ be a semiorthogonal decomposition and let $\sfS_{\sC}$ be the Serre functor of $\sC$. Then 
\begin{equation*}
    \sC \simeq \langle \sfS_{\sC}(\sC_2) , \sC_1 \rangle \simeq \langle \sC_2, \sfS_{\sC}^{-1}(\sC_1) \rangle
\end{equation*}
are also semiorthogonal decompositions.
\end{proposition}

\begin{proposition} \label{prop:serre_subcategory}
Let $\sC = \langle \sC_1, \sC_2 \rangle$. Then 
\[ \sfS^{-1}_{\sC_1} = \bL_{\sC_2} \circ S^{-1}_{\sC}  \quad \text{and} \quad \sfS_{\sC_2} = \bR_{\sC_1} \circ \sfS_{\sC}. \]
\end{proposition}

\subsection{Equivariant triangulated categories} \label{sec:equivariant_triangulated_categories}

The definitions of a group action on a category and the corresponding equivariant category are due to Deligne \cite{deligne1997action}. The following background can be found in \cite[Section 3]{kuznetsov2017derived}, which we follow. Let $G$ be a finite group and $\sC$ a triangulated category.

\begin{definition} \label{def:categorical_action}
A \emph{(right) action} of $G$ on $\sC$ is given by the following data:
\begin{enumerate}
    \item for each $g \in G$, an autoequivalence $g^* \colon \sC \to \sC$;
    \item for each pair $g, h \in G$, a natural isomorphism $c_{g,h} \colon (gh)^* \xrightarrow{\sim} h^* \circ g^*$ such that
    \[\begin{tikzcd}
	{(fgh)^*} & {h^* \circ (fg)^*} \\
	{(gh)^*\circ f^*} & {h^* \circ g^* \circ f^*}
	\arrow["{c_{fg,h}}", from=1-1, to=1-2]
	\arrow["{c_{f,gh}}"', from=1-1, to=2-1]
	\arrow["{c_{g,h}f^*}", from=2-1, to=2-2]
	\arrow["{h^*c_{f,g}}", from=1-2, to=2-2]
\end{tikzcd}\]
commutes for all $f,g,h \in G$.
\end{enumerate}
\end{definition}

\begin{definition}
    A \emph{$G$-equivariant object} of $\sC$ is a pair $(F, \phi)$ consisting of an object $F \in \sC$ and a collection of isomorphisms $\phi_g \colon F \xrightarrow{\sim} g^* (F)$ for all $g \in G$ such that the diagram
    \[\begin{tikzcd}
	F & {h^*(F)} & {h^*(g^*(F))} \\
	&& {(gh)^*(F)}
	\arrow["{h^*(\phi_g)}", from=1-2, to=1-3]
	\arrow["{c_{g,h}(F)}"', from=2-3, to=1-3]
	\arrow["{\phi_{gh}}"', from=1-1, to=2-3]
	\arrow["{\phi_h}", from=1-1, to=1-2]
\end{tikzcd}\]
commutes for all $g,h \in G$. The isomorphisms $\phi=\{\phi_g\}_{g\in G}$ are called the \emph{$G$-linearisation}. The \emph{$G$-equivariant category} $\sC^G$ of $\sC$ is the category whose objects are the $G$-equivariant objects of $\sC$, and morphisms are those between $G$-invariant objects of $\sC$ that commute with the $G$-linearisations. 
\end{definition}

\begin{example}[{\cite[p. 12]{elagin2014equivariant}}]
    Suppose $G$ is a finite abelian group. Let $\widehat{G}=\Hom(G,\bC^\times)$ be the group of irreducible representations of $G$. Then there is an action of $\widehat{G}$ on $\sC^G$ given as follows. For every $\rho\in\widehat{G}$, we have an autoequivalence
	\begin{equation*}
		\rho^\ast(F,(\phi_h)):= (F,(\phi_h))\otimes \rho := (F,(\phi_h\cdot \rho(h))).
	\end{equation*}
	For $\rho_1$,$\rho_2\in \widehat{G}$, the equivariant objects $\rho_1^\ast\rho_2^\ast(F,(\phi_h))$ and $(\rho_2\circ \rho_1)^\ast(F,(\phi_h))$ are the same, hence we set the isomorphisms $c_{\rho_2,\rho_1}$ to be the identities.
\end{example}

\begin{theorem}[{\cite[Theorem 6.3]{elagin2012descent}, \cite[Proposition 3.11]{elagin2014equivariant}, \cite[Theorem 3.2]{kuznetsov2017derived}}]\label{thm:equivariant_sod}
    Let $X$ be a quasi-projective variety with an action of a finite group $G$. Suppose we have a semiorthogonal decomposition $\Db(X) = \langle \sA_1, \dots, \sA_n \rangle$ where $g^\ast\sA_i\subset\sA_i$ for all $g\in G$. Then the $G$-equivariant category has the following semiorthogonal decomposition
    \begin{equation*}
        \Db(X)^G = \langle \sA_1^G , \dots, \sA_n^G \rangle.
    \end{equation*}
\end{theorem}

\begin{definition}
    An action of $G$ on $\sC$ is called \emph{trivial} if for each $g \in G$ there is an isomorphism of functors, $\tau_g \colon \id \xrightarrow{\sim} g^*$, such that $c_{g,h} \circ \tau_{gh} = h^* \tau_g \circ \tau_h$ for all $g,h \in G$.
\end{definition}

\begin{proposition}[{\cite[Proposition 3.3]{kuznetsov2017derived}}] \label{prop:orthogonal_equivariant_sod}
    Let $G$ be a finite group acting trivially on a triangulated category $\sC$. Then $\sC^G$ is also triangulated. Let $\rho_0, \dots, \rho_n$ be the irreducible representations of $G$. Then there is a completely orthogonal\footnote{Meaning the Homs vanish in both directions.} decomposition 
    \begin{equation*}
        \sC^G = \langle \sC \otimes \rho_0, \dots , \sC \otimes \rho_n \rangle.
    \end{equation*}
\end{proposition}

In the situation of Proposition \ref{prop:orthogonal_equivariant_sod}, we define the functors 
\begin{align}
    \iota_k &\colon \sC \to \sC^G, \quad F \mapsto F \otimes \rho_k; \\
    \pi_k &\colon \sC^G \to \sC, \quad F \mapsto F \otimes_{\bC[G]} \rho_k^\vee .
\end{align}
The functors $\pi_k$ are both left and right adjoint to $\iota_k$.

\section{Cyclic covers and semiorthogonal decompositions} \label{sec:covers}

\subsection{Cyclic covers and rectangular Lefschetz decompositions}\label{ssec:setup_covers}
In this section we follow the notation of \cite{kuznetsov2017derived}. Let $Y$ be an algebraic variety (or a proper Deligne--Mumford (DM) stack). Consider $f\colon X \rightarrow Y$, the degree $n$ cyclic cover of $Y$ ramified over a divisor $Z$.
Denote by $j \colon Z \hookrightarrow X$ the embedding of $Z$ as the ramification divisor.

Let $\mu_n$ denote the group of $n$-th roots of unity. This has dual group $\widehat{\mu_n}\coloneqq\Hom(\mu_n,\bC^\times)\simeq\bZ/n$.
This isomorphism is given by the primitive character $\chi \colon \mu_n \to \bC^\times$.
Each irreducible representation $\rho_i$ of $\mu_n$ corresponds to $\chi^i$.

We consider the trivial action of $\mu_n$ on $Y$. Then $f$ is $\mu_n$-equivariant, and induces the following functors between equivariant derived categories:
\begin{align*}
    f_k^* &\colon \Db(Y) \xrightarrow{\iota_k} \Db(Y)^{\mu_n} \xrightarrow{f^*} \Db(X)^{\mu_n}, \\
    f_{k *} &\colon \Db(X)^{\mu_n} \xrightarrow{f_*} \Db(Y)^{\mu_n} \xrightarrow{\pi_k} \Db(Y), \\
    f_k^! &\colon \Db(Y) \xrightarrow{\iota_k} \Db(Y)^{\mu_n} \xrightarrow{f^!} \Db(X)^{\mu_n} .
\end{align*}
Note that $f_k^*$ is left adjoint to $f_{k *}$, and $f_k^!$ is right adjoint to $f_{k *}$. Similarly, define the functors $j_k^*, j_{k *}, j_k^!$ as the compositions
\begin{align*}
    j_k^* &\colon \Db(X)^{\mu_n} \xrightarrow{j^*} \Db(Z)^{\mu_n} \xrightarrow{\pi_k} \Db(Z), \\
    j_{k *} &\colon \Db(Z) \xrightarrow{\iota_k} \Db(Z)^{\mu_n} \xrightarrow{j_*} \Db(X)^{\mu_n}, \\
    j_k^* &\colon \Db(X)^{\mu_n} \xrightarrow{j^!} \Db(Z)^{\mu_n} \xrightarrow{\pi_k} \Db(Z) .
\end{align*}
Also note that $j_k^*$ is left adjoint to $j_{k *}$, and $j_k^!$ is right adjoint to $j_{k *}$.

\begin{theorem}[{\cite[Theorem 4.1]{kuznetsov2017derived}}]\label{prop:sod_2}
    For each $k\in \bZ/n$, the functors $f^\ast_k$ and $j_{k\ast}$ are fully faithful. Moreover, there is a semiorthogonal decomposition
    \begin{equation*}
        \Db(X)^{\mu_n}=\langle  f_0^* \Db(Y), j_{0*} \Db(Z) , \dots,  j_{n-2*} \Db(Z) \rangle.
    \end{equation*}
\end{theorem}

Now additionally assume that $Y$ has a rectangular Lefschetz decomposition, i.e. assume that there is a line bundle $\sO_Y(1)$ and an admissible subcategory $\sB \subset \Db(Y)$ such that
\begin{equation*}
    \Db(Y)=\langle \sB, \sB(1),\ldots,\sB(m-1)\rangle
\end{equation*}
is a semiorthogonal decomposition, where $\sB(t)\coloneqq \sB\otimes \sO_Y(t)$. Assume moreover that $Z\in |\sO_Y(nd)|$, where $n,d$ are positive integers satisfying $0<m-(n-1)d\eqqcolon M$. Then \cite[Lemma 5.1]{kuznetsov2017derived} applies, hence $f^*$ is fully faithful and $\Db(X)$ has the semiorthogonal decomposition 
\begin{equation} \label{eq:ordinary_kuznetsov_component_sod}
    \Db(X) = \langle \sA_X, \sB_X, \dots, \sB_X(M-1) \rangle,
\end{equation}
where $\sB_X \coloneqq f^* \sB$. We call $\sA_X = \langle \sB_X, \dots, \sB_X(M-1) \rangle^\perp$ the \emph{Kuznetsov component}.

For subsets $T \subset \bZ$ and $S \subset \bZ/n$, define
\begin{equation*}
    \sB_{X}^{S}(T) \coloneqq \langle \sB_X(t) \otimes \rho_k \rangle_{t \in T, k \in S} \subset \Db(Y)^{\mu_n}.
\end{equation*}

\begin{proposition} \label{prop:sod_1}
    We have the following semiorthogonal decomposition,
    \begin{equation*}
        \Db(X)^{\mu_n}=\langle \sA_X^{\mu_n}, \sB_X^{[0,n-1]}([0,M-1])\rangle.
    \end{equation*}
\end{proposition}

\begin{proof}
    Each piece of the semiorthogonal decomposition (\ref{eq:ordinary_kuznetsov_component_sod}) is preserved by the action of $\mu_n$. Thus, by Theorem \ref{thm:equivariant_sod} the action of $\mu_n$ distributes through the semiorthogonal decomposition. Furthermore, since the action of $\mu_n$ is trivial on each piece $\sB_X(i)$, we get $\sB_X(i)^{\mu_n} = \sB_X^{[0,n-1]}(i)$ by Proposition \ref{prop:orthogonal_equivariant_sod}. The result follows. 
\end{proof}

We also state a useful lemma which we will use in the next section.

\begin{lemma}[{\cite[Lemma 6.1]{kuznetsov2017derived}}] \label{lem:right_mutation_lemma}
    For any twist $ t \in \bZ$ and weight $k \in \bZ/n$ we have
    \begin{equation*}
        \bR_{j_{k *} \Db(Z)} (\sB_X^k(t)) = \sB_X^{k+1}(t-d) .
    \end{equation*}
\end{lemma}

\subsection{A semiorthogonal decomposition of \texorpdfstring{$\sA_X^{\mu_n}$}{the equivariant Kuznetsov component}}
\label{sec:sodAX}

For this section, assume moreover that $0<M<d$. Before stating the main theorem, let us set up some more notation. Define the functor $\Phi_k \colon \Db(X)^{\mu_n} \to \Db(X)^{\mu_n}$ for $0 \leq k \leq n-2$ as 
\begin{equation*}
    \Phi_k(-) \coloneqq \bL_{\sB_X^{[0,k]}([0,M-1])}(- \otimes \rho_k).
\end{equation*}
Furthermore, define the subcategory $\sA_Z$ of $\Db(X)^{\mu_n}$ as follows:
\begin{equation} \label{eq:A_Z_definition}
    \sA_Z \coloneqq \langle j_{0 *} \Db(Z), \sB_X^1([M-d,-1]) \rangle .
\end{equation}

\begin{theorem} \label{thm:main}
    Let $X$ and $Z$ be as in Section \ref{ssec:setup_covers}, with $0<M<d$. We have a semiorthogonal decomposition,
    \begin{equation*}
        \sA_X^{\mu_n} = \langle \Phi_0(\sA_Z) , \Phi_1(\sA_Z) , \dots, \Phi_{n-2}(\sA_Z) \rangle .
    \end{equation*}
\end{theorem}

\begin{remark}
The results of \cite{kuznetsov2017derived} apply to the numerical range $M\geq d$, while we work with $0<M<d$. See Section \ref{relation to other works} for a comparison between the results.
\end{remark}

\begin{proof}
    We follow the strategy of the proof in Section 6 of \cite{kuznetsov2017derived}. From Theorem \ref{prop:sod_2}, we have the semiorthogonal decomposition
    \begin{equation} \label{eq:starting_sod}
        \Db(X)^{\mu_n} = \langle \sB_X^0([0,m-1]), j_{0*} \Db(Z) , \dots,  j_{n-2*} \Db(Z) \rangle.
    \end{equation}
    Now rewrite the first component of the semiorthogonal decomposition above as 
    \begin{equation}
    \begin{split}\label{eq:B_split_up}
        \sB_X^0([0,m-1]) = \langle &\sB_X^0([0,M-1]) , \sB_X^0([M, M+d-1]), \\ &\sB_X^0([M+d, M+2d-1]), 
        \dots , \sB_X^0([m-d, m-1]) \rangle .
        \end{split}
    \end{equation}
    We next substitute the decomposition (\ref{eq:B_split_up}) into the semiorthogonal decomposition (\ref{eq:starting_sod}), and iteratively apply right mutations together with Lemma \ref{lem:right_mutation_lemma}. Then, as in \cite{kuznetsov2017derived}, we get
    \begin{equation}
        \begin{split}
            \label{eq:DXmu_split_up}
        \Db(X)^{\mu_n} = \langle &\sB_X^0([0,M-1]), j_{0*}\Db(Z), \sB_X^1([M-d,M-1]), j_{1*} \Db(Z), \\
        &\sB_X^2([M-d, M-1]), \dots, j_{n-2 *} \Db(Z), \sB_X^{n-1}([M-d, M-1]) \rangle .
        \end{split}
    \end{equation}
    Now we deviate from the proof in \emph{loc. cit.} and use $\sA_Z$ as defined above. Since $M<d$, \[\sB_X^k([M-d,M-1]) = \langle \sB_X^k([M-d,-1]), \sB_X^k([0,M-1]) \rangle \] for any  weight $k$. Also, since $j\colon Z \to X$ is equivariant, $j_{l*}(-)\otimes \rho_k = j_{l+k *}$. Therefore, \eqref{eq:DXmu_split_up} reads:
\begin{equation}
    \begin{split}
         \Db(X)^{\mu_n} = \langle &\sB_X^0([0,M-1]), \\
         & \sA_Z, \sB_X^1([0,M-1]), \\
         & \sA_Z \otimes \rho_1, \sB_X^2([0, M-1]),\\
         &\dots, \\
         &\sA_Z \otimes \rho_{n-2}, \sB_X^{n-1}([0, M-1]) \rangle.
    \end{split}
\end{equation}
We next apply left mutations and regroup as follows.
    \begin{equation}
        \begin{split}
        \Db(X)^{\mu_n} = \langle &\bL_{\sB_X^0([0,M-1])} ( \sA_Z) ) , \\
        &\bL_{\sB_X^0([0,M-1])}\bL_{\sB_X^1([0,M-1])}(\sA_Z \otimes \rho_1 ), \\
        &\dots,  \\
            &\bL_{\sB_X^0([0,M-1])}\bL_{\sB_X^1([0,M-1])}  \cdots  \bL_{\sB_X^{n-2}([0,M-1])} ( \sA_Z\otimes \rho_{n-2} ) ,  \\
            &\sB_X^0([0,M-1]), \sB_X^1([0,M-1]), \dots, \sB_X^{n-1}([0,M-1]) \rangle \\
            =  \langle &\bL_{\sB_X^0([0,M-1])} (\sA_Z ) , \\
            & \bL_{\sB_X^{[0,1]}([0,M-1])} ( \sA_Z \otimes \rho_1 ) , \dots, \\
            &\bL_{\sB_X^{[0,n-2]}([0,M-1])} (\sA_Z \otimes \rho_{n-2} ) , \sB_X^{[0,n-1]}([0,M-1]) \rangle \\
        =\langle & \Phi_0(\sA_Z) , \Phi_1(\sA_Z) , \dots, \Phi_{n-2}(\sA_Z), \sB_X^{[0,n-1]}([0,M-1])\rangle.
        \end{split}
    \end{equation}
    Since the right-hand sides of  the semiorthogonal decomposition above and in Proposition \ref{prop:sod_1} match up, the components to the left are equivalent. The statement of the theorem follows.
\end{proof}

\begin{corollary}
    Under the action of $\widehat{\mu}_n\simeq\bZ/n$ on $\Db(X)^{\mu_n}$, we have
        \begin{equation*}
        \sA_X = \langle \Phi_0(\sA_Z) , \Phi_1(\sA_Z) , \dots, \Phi_{n-2}(\sA_Z) \rangle^{\bZ / n} .
    \end{equation*}
\end{corollary}

\begin{proof}
    By \cite[Theorem 4.2]{elagin2014equivariant}, $(\sA_X^{\mu_n})^{\bZ/n}\simeq \sA_X$. The result then follows by Theorem \ref{thm:main}.
\end{proof}

\begin{example}\label{ex:sods}
Theorem \ref{thm:main} applies to the following classes of examples.
\begin{enumerate}
    \item The Veronese double cone $Y_1$, defined as the double cover of $\bP(1,1,1,2)$ branched over $Z\in |\sO_{\bP(1,1,1,2)}(6)|$. This is a smooth prime Fano threefold of index $2$ and degree $1$. In this case, $\sA_Z = \langle \Db(Z), \Db(\pt) \rangle.$ l\label{itm:veronese}
    \item The prime Fano threefold $X_2$ of index $1$ and genus $2$, which is realized as a double cover of $\bP^3$ branched in a sextic. Then $\sA_Z=\langle \Db(Z), \Db(\pt),\Db(\pt) \rangle$. \label{itm:X2}
\end{enumerate}
\end{example}

\begin{example}[The family $X_4$ and smoothings of $\frac{1}{4}(1,1)$]\label{ex:X4andStack}
Consider a Fano threefold $X$ obtained as the double cover $f\colon X \to Q$ of a quadric hypersurface $Q$ in $\bP^4$ branched over a section of $\sO_Q(2)$. The derived category of $Q$ admits a strong full exceptional collection $\Db(Q)=\langle S', \sO_Q,\sO_Q(1),\sO_Q(2) \rangle$, where $S'$ is a spinor bundle on $Q$ \cite{Kapranov}. This decomposition is not rectangular. Hence, the present setting is not covered by the results of this section. Nevertheless, we can apply the same methods and compare $\sA_X^{\mu_2}$ and $\Db(Z)$. 
Denote by $\sC = \langle S',\sO_Q,\sO_Q(1) \rangle$. By Theorem \ref{prop:sod_2}, and then arguing as in Lemma \ref{lem:right_mutation_lemma}, we obtain a decomposition
\begin{equation}
    \begin{split}
        \Db(X)^{\mu_2} &=  \langle f^*_0\sC, \sO_X(2) \otimes \rho_0, j_{0*}\Db(Z) \rangle \\
        &= \langle f^*_0\sC, j_{0*}\Db(Z), \sO_X \otimes \rho_1 \rangle.
    \end{split}
\end{equation}

On the other hand, $\sO_X=f^\ast \sO_Q$ is exceptional by Kodaira vanishing. Let $\sA_X\coloneqq \langle \sO_X \rangle^\perp$, so $\Db(X)=\langle \sA_X, \sO_X\rangle$.
Proposition \ref{prop:sod_1} then gives
\[ \Db(X)^{\mu_2} =  \langle \sA_X^{\mu_2}, \sO_X \otimes \rho_0, \sO_X \otimes \rho_1 \rangle. \]
Comparing the two decompositions we obtain
\begin{equation}\label{eq:sod_X4}
    \langle \sA_X^{\mu_2}, \sO_X \otimes \rho_0 \rangle = \langle f^*_0\sC, j_{0*}\Db(Z) \rangle.
\end{equation}
The above equality can be used to relate the K-theories of $\sA_X^{\mu_2}$ and $\Db(Z)$ (see Remark \ref{rmk:X4Hodge}).

There exists another family of double covers, i.e. double covers of $\bP(1,1,1,2)$ branched over a section $Z$ of $\sO_{\bP(1,1,1,2)}(8)$, which can be equivalently regarded as quartic hypersurfaces of $\bP\coloneqq \bP(1,1,1,2,4)$. Let $X$ be one such hypersurface. If $X$ is chosen to miss the points $[0:0:0:1:0]$ and $[0:0:0:0:1]$, then $X$ has two singularities $\frac 12(1,1,1)$ in its intersection with the line $x_0=x_1=x_2=0$, and is smooth otherwise.
Now consider the canonical stack $\sX$ associated to $X$: it has index $M=1$ and by Theorem  \ref{thm:main} there is a semiorthogonal decomposition
\[ \sA_\sX^{\mu_2} = \langle \Db(Z), \Db(\pt),\Db(\pt),\Db(\pt) \rangle. \]
We point out that the family of threefolds $X_4$ and that of hypersurfaces just constructed admit a common degeneration as follows. It is well known that the cone over a rational normal quartic curve, i.e. the plane $\bP(1,1,4)$, has smoothings both to a Veronese surface and to a quadric surface \cite[14,~Example~4]{Ste03}. Then, a cone over $\bP(1,1,4)$ may be deformed to a singular quadric threefold and to $\bP(1,1,1,2)$, which is a cone over a Veronese surface. Thus, $X_4$ may be degenerated to the double cover of a singular quadric, which in turn degenerates to a cone over $\bP(1,1,4)$.
\end{example}

\subsection{Comparison to \cite{kuznetsov2017derived}}  \label{relation to other works}

The article \cite{kuznetsov2017derived} works under the assumption $M\geq d$.
If the inequality is strict, $\Db(Z)$ admits a Kuznetsov component and a decomposition as follows. Let $\sB_Z$ be the essential image of the restriction along the inclusion $Z\subset Y$. Then
    \begin{equation}
    \label{eq:KP_Z_sod}
    \Db(Z) = \langle \sA_Z^{\mathsf{KP}}, \sB_Z, \dots, \sB_Z(M-d-1) \rangle,
\end{equation}
where $\sA_Z^{\mathsf{KP}}$ denotes here the right orthogonal, i.e. the Kuznetsov component of $\Db(Z)$, in contrast with \eqref{eq:A_Z_definition}. If $M=d$, then $\Db(Z) = \sA_Z^{\mathsf{KP}}$ (see \cite[Lemma 5.5]{kuznetsov2017derived}).

In \cite[Theorem 1.1]{kuznetsov2017derived} the authors show that for $M \geq d$ the $\mu_n$-equivariant Kuznetsov component, $\sA_X^{\mu_n}$, has a semiorthogonal decomposition into $n-1$ copies of $\sA_Z^{\mathsf{KP}}$.

Theorem \ref{thm:main} shows that, when $0<M<d$, there exists a category $\sA_Z$, defined as a subcategory of $\Db(X)^{\mu_n}$ (see (\ref{eq:A_Z_definition})), and a similar decomposition of $\sA_X^{\mu_n}$. 
Observe that $\sA_Z$ now \textit{contains} $\Db(Z)$ as a semiorthogonal component. 

    Suppose $Y$ is weighted projective space. Then $Y$ has a rectangular Lefschetz decomposition with $\sB=\langle \sO_Y\rangle$, and $\omega_Y\simeq \sO_Y(-m)$. By adjunction, $\omega_Z \simeq \sO_Z(nd-m)$,. Hence the sign of $nd-m = d-M$ determines whether $Z$ is Fano, K-trivial, or canonically polarized. The first two cases correspond to $M\geq d$ and are those considered in \cite{kuznetsov2017derived}. In summary,

\begin{center}
\begin{tikzcd}[row sep=0.1em, column sep =0.5em]
d\leq M & {\substack{\text{\cite[Lemma 5.5]{kuznetsov2017derived}}\\ \implies}} & \sA_Z^{\mathsf{KP}} \hookrightarrow \Db(Z), \\
0<M<d   & \substack{\text{Section \ref{sec:sodAX}} \\ \implies}                                 & \Db(Z) \hookrightarrow \sA_Z.             
\end{tikzcd}
\end{center}

\section{Equivariant equivalences and Hodge isometries}\label{ssec:equivar_equivalences}

\subsection{Cohomology of varieties and Hodge theory}

Let $X$ be a smooth projective variety over $\bC$ of dimension $n$. For all $k$, the singular cohomology group $H^k(X, \bZ)$ carries a Hodge structure. The complexification $H^k(X, \bC)$ decomposes as
\begin{equation*}
    H^k(X, \bC) = \bigoplus_{p+q=k} H^{p,q}(X),
\end{equation*}
where $H^{p,q}(X) \simeq H^q(X, \Omega_X^p)$, for all $p+q=k\geq 0$. We have $H^{p,q}(X) = \overline{H^{q,p}(X)}$.

Singular cohomology has the structure of a commutative ring with cup product. The class of a hyperplane section $H$ of $X$ induces the operation $- \cup H \colon H^k(X, \bZ) \to H^{k+2}(X, \bZ)$. Then, \textit{primitive cohomology} is defined as 
\begin{equation*}
    H^{n-k}_\prim(X, \bZ) \coloneqq \ker(-\cup H^{k+1} \colon H^{n-k}(X, \bZ) \to H^{n+k+2}(X, \bZ) )
\end{equation*}
for all $k$. The Hodge structure on $H^{n-k}(X,\bZ)$ restricts to one on $H^{n-k}_\prim(X, \bZ)$ \cite[II.6-7]{Voisin}.

\subsection{Topological K-theory}\label{ssec:TopKthy}

Our main references for topological K-theory are \cite{weibel} and the papers \cite{blanc2016topological}, \cite[Section 5.1]{perry2022integral}, and \cite[Section 2]{DHLncHDR}. We collect the results needed for this work below.

There is a lax monoidal functor 
\[\rK^\top  \colon \mathrm{Cat}_\bC \to \mathrm{Sp}\]
from $\bC$-linear categories to the $\infty$-category of spectra, which satisfies the following properties.

\begin{theorem}[{\cite{blanc2016topological}}] \label{thm:blanc} \leavevmode
    \begin{enumerate}
        \item If $\sC=\langle \sC_1,\ldots,\sC_m \rangle$ is a $\bC$-linear semiorthogonal decomposition, then there is an equivalence 
\begin{equation}\label{eq:KtopSOD}
        \rK^\top (\sC) \simeq \rK^\top (\sC_1) \oplus \ldots \oplus \rK^\top (\sC_m),
        \end{equation}
        where the map $\rK ^\top(\sC)\to \rK^\top (\sC_i)$ is induced by the projection functor to $\sC_i$. \label{itm:Blanc1}
        \item Let $\rK(\sC)$, $\mathrm{HN}(\sC)$, and $\HP(\sC)$ denote algebraic $\rK$-theory, negative cyclic homology, and  periodic cyclic homology of $\sC$ respectively. There is a functorial commutative square
        \begin{equation}\label{eq:KtopKalgCD}
        \begin{tikzcd}
            \rK(\sC) \rar{} \dar & \mathrm{HN}(\sC) \dar \\
            \rK^\top (\sC) \rar{} & \mathrm{HP}(\sC)
        \end{tikzcd}.    
        \end{equation} 
        \item If $X$ is a proper complex variety, then there exists a functorial equivalence $\rK^\top (\Dperf(X)) \simeq \rK^\top (X)$, the complex $\rK$-theory spectrum of $X$. Under this equivalence, the left vertical arrow in \eqref{eq:KtopKalgCD} recovers the usual inclusion of algebraic $\rK$-theory into topological $\rK$-theory, and the bottom arrow coincides with the usual Chern character under the identification of $\HP(\Dperf(X))$ with 2-periodic de Rham cohomology.
        \end{enumerate}
\end{theorem}

For $\sC \in \Cat_{\bC}$ and $t$ an integer, write $\rK_t^\top(\sC) = \pi_t(\rK^\top (\sC))$.
If in addition the category $\sC$ is proper, then for each integer $t$ there is a bilinear Euler pairing $\chi^\top\colon \rK_t^\top(\sC) \otimes \rK_t^\top(\sC) \to \bZ$ induced by the evaluation functor 
\[ \sH\mathrm{om}(-,-)\colon \sC^{\mathrm{op}} \otimes_{\Dperf(\Spec(\bC))} \sC \to \Dperf(\Spec(\bC)), \] where $\Dperf$ denotes the category of perfect complexes. The Euler pairing satisfies the following properties \cite[Lemma 5.2]{perry2022integral}:
\begin{itemize}
    \item The pairing is compatible with the inclusions of $\rK_t^\top(\sC_i)$, and it makes \eqref{eq:KtopSOD} into a semiorthogonal sum (i.e. $\chi^\top(v_i,v_j)=0$ for $v_i\in \rK_t^\top(\sC_i)$, $v_j\in \rK_t^\top(\sC_j)$, $i>j$);
    \item when restricted to $\rK_0(\sC)$ via the inclusion in \eqref{eq:KtopKalgCD}, $\chi^\top$ coincides with the usual Euler pairing 
    \[\chi(E,F)=\sum_{i\in \bZ}(-1)^i \dim \Hom^i_\bC(E,F).\]
    \item (Riemann--Roch) If $\sC=\Dperf(X)$, with $X$ a proper complex variety, then for $v,w\in \rK_t^\top(\sC)$ we have $\chi^\top(v,w)=p_*(v^\vee \otimes w) \in \rK_{2t}^\top(\Spec(\bC))\simeq \bZ$, with $p\colon X \to \Spec(\bC)$ the structure morphism.
\end{itemize}

Now we suppose $\sC$ is an admissible subcategory of the derived category of a proper smooth DM quotient stack, to leverage results on the non-commutative Hodge-to-de Rham spectral sequence. We say that $\rK^\top (\sC)$ carries a \textit{pure Hodge structure} if every homotopy group $\rK_{t}^\top(\sC)$ carries a pure Hodge structure of weight $-t$. 

\begin{proposition}\label{prop:HS_C_in_Stack}
Let $\sC$ be an admissible subcategory of $\Dperf(\cX)$, for $\cX=[X/G]$ a proper smooth DM stack.
\begin{enumerate}
    \item $\rK^\top (\sC)$ comes endowed with a canonical pure Hodge structure with graded pieces 
    \begin{equation}
\mathrm{gr}^p(\rK_t^\top(\sC)_\bC) \simeq \mathrm{HH}_{t+2p}(\sC),      
    \end{equation}
 where $\HH_{t}(\sC)$ denotes the $t$-th Hochschild homology group of $\sC$.
\item For $\sC = \Dperf(\cX)$, the Chern character induces an isomorphism of Hodge structures
\begin{equation}\label{eq:KPerfHS}
    \rK_t^\top(\Dperf(\cX))_\bQ \simeq \bigoplus_{k \in \bZ} H^{2k - t}(I_{\cX}, \bQ)(k), 
\end{equation}
where $I_{\cX}$ denotes the (underived) inertia stack, and the summands on the right-hand side are its de Rham cohomology groups.
\item If $\cX=X$ is a smooth proper complex variety, the summands in \eqref{eq:KPerfHS} coincide with $H^{2k-t}(X,\bQ)(k)$, the $k$-th Tate twists of rational singular cohomology groups.
\end{enumerate}
\end{proposition}

\begin{proof}
To prove these statements, we claim that the map $\rK^\top (\sC)_\bC \to \HP(\sC)$ in \eqref{eq:KtopKalgCD} is an isomorphism (this is referred to as the \textit{lattice property for $\sC$}), and that the non-commutative Hodge-to-de Rham sequence degenerates for $\sC$ (we say that $\sC$ has the \textit{degeneration property}). Granting the claims, the degeneration property gives a canonical filtration of $\HP_n(\sC)$ with $p$-th graded piece $\HH_{n+2p}(\sC)$, which together with the lattice property shows (1). 

There are identifications $\mathrm{HP}_t(\Dperf(\cX)) \simeq \oplus_{k \in \bZ} H^{2k - t}(I_{\cX}, \bQ)$ by \cite[Proposition 2.13]{DHLncHDR}, and $\mathrm{HP}_t(\Dperf(X))\simeq H^{2k-t}(X,\bQ)$ by \cite{weibel}, under which the non-commutative Hodge filtration coincides with the usual Hodge filtrations. This proves parts (2) and (3).

To establish the claims, observe first that they are preserved under direct summands and arbitrary direct sums in the category of noncommutative motives (defined in \cite{tabuada}). Indeed, this holds for the lattice property because of the functoriality of \eqref{eq:KtopKalgCD}, and for the degeneration property by \cite[Lemma 1.22]{DHLncHDR}. Therefore, it is sufficient to establish them in the case $\sC=\Dperf(\cX)$, for which they are proven, respectively, in \cite[Corollary 2.19]{DHLncHDR} and \cite[Corollary 1.23]{DHLncHDR}.
\end{proof}

\subsection{Equivalences of Fourier--Mukai type induce Hodge isometries}

Consider smooth proper DM stacks that are global quotients, $[X_j/G_j]$, for $j=1,2$.
Fix admissible subcategories $i_1\colon\sK_1 \to \Db(X_1)^{G_1}$ and $i_2\colon\sK_2 \to \Db(X_2)^{G_2}$, and let $i_j^*$ denote the left adjoints to the inclusions for $j=1,2$.
Recall that an equivalence $\phi \colon \sK_1 \to \sK_2$ is said to be of \textit{Fourier--Mukai type} if the composition $\psi \colon \Db(X_1)^{G_1} \to \Db(X_2)^{G_2}$ given by
\begin{equation}
    \label{eq:FMtypeComp}
    \Db(X_1)^{G_1} \xrightarrow{i_1^*} \sK_1 \xrightarrow{\phi} \sK_2 \xrightarrow{i_2} \Db(X_2)^{G_2}
\end{equation}

is a Fourier--Mukai functor.
Denote by $(\sE_j,e_j)$ a dg-enhancement for each $\Db(X_j)^{G_j}$, which in turn induce enhancements $(\sF_j,f_j)$ for $\sK_j$. In this setting:

\begin{lemma}\label{lem:lift}
An equivalence $\phi \colon \sK_1 \to \sK_2$ of Fourier--Mukai type lifts to an equivalence of dg-enhancements $(\sF_1,f_1) \to (\sF_2,f_2)$. 
\end{lemma}

\begin{proof}
    The Fourier--Mukai assumption means that the composition $\psi$ in \eqref{eq:FMtypeComp}
    is of Fourier--Mukai type, whence it lifts to a functor $\Psi\colon (\sE_1,e_1) \to (\sE_2,e_2)$ (this is \cite[Theorem 8.9]{toen2007homotopy} for schemes, and \cite[Corollary~3.7]{Kung22} for the equivariant version). Let $i_j^{\mathrm{dg}}$ and $i_j^{*\mathrm{dg}}$ denote the dg-lifts of the inclusions and their left adjoint functors. Define $\Phi \coloneqq i_2^{*\mathrm{dg}}\circ \Psi \circ i_1^{\mathrm{dg}}$. The functor $\Phi \colon (\sF_1,f_1) \to (\sF_2,f_2)$ lifts $\phi$. In fact, taking cohomology, we have
     \[H^0(\Phi) \simeq  i_2^{*}\circ H^0(\Psi) \circ i_1 \simeq i_2^{*}\circ i_2 \circ \phi \circ i_1^* \circ i_1 \simeq \phi. \qedhere\]
\end{proof}

Observe that Proposition \ref{prop:HS_C_in_Stack} applies to each $\sK_j$, and endows each $\rK^\top (\sK_j)$ with Hodge structures. Then we have the following Proposition.

\begin{proposition}\label{prop:k_theory_hs_preserved}
    A Fourier--Mukai type equivalence $\phi \colon \sK_1 \to \sK_2$ induces a Hodge isometry $\rK^\top (\sK_1)\simeq \rK^\top (\sK_2)$.
\end{proposition}

\begin{proof}
By the functoriality of $\rK^\top $, and since the Euler pairing is defined through a canonical evaluation map, we immediately get isometries $\rK^\top (\sK_1)\simeq \rK^\top (\sK_2)$.

Since $\phi$ is of Fourier--Mukai type, it admits a lift $\Phi\colon (\sF_1,f_1)\to (\sF_2,f_2)$ to the dg-enhancements by Lemma \ref{lem:lift}. The construction of the Hodge filtration on $\rK^\top $ only depends on the noncommutative Hodge-to-de Rham spectral sequence, which is a motivic invariant of the dg-enhancements.  
Then $\rK^\top (\sK_1)\simeq \rK^\top (\sK_2)$ is a Hodge isometry.
\end{proof}

\section{Application: weighted double solids}\label{sec:doublesolids}

In this section we focus on the three-dimensional case. 

\subsection{Setup} We collect here our assumptions and notation.

\begin{definition}
    \label{setup}
We say that $X$ is a \textit{weighted double solid} if $X$ is a three-dimensional smooth DM stack equipped with a 2:1 map to a weighted projective space $Y$ (regarded as a smooth stack), branched over a divisor $Z\in |\sO_Y(2d)|$. We say that $X$ is \textit{prime} if it has Picard rank 1.
\end{definition}

 In the notation of $\ref{ssec:setup_covers}$, a weighted double solid satisfies $n=2$ and $M=m-d$. Here, $m$ is the sum of the weights of the coordinates of $Y$. We will sometimes need an additional generality assumption, which ensures that $Z$ has Picard rank 1 (this holds for $X$ prime and very general). We refer to $X$ in this case as a \textit{very general} weighted double solid.

\begin{remark} \label{rem:X_is_Fano} 
A weighted double solid with $0<M$ is Fano. Indeed, denote by $\sO_X(1)$ the (ample) pull-back of $\sO_Y(1)$. Then, by the Riemann--Hurwitz formula we have $K_X = f^* K_{Y} + \sO_X(d) = \sO_X(-m + d) = \sO_X(-M)$, and $M>0$ is the index of $X$. Similarly, $K_Z = \sO_Z(-m + 2d)$, so $Z$ is canonically polarized if $M<d$ and $K$-trivial if $M=d$.
\end{remark}

Koll\'ar's \cite[Theorem~29]{KollarHyper} implies the following Lemma \ref{lem:ram_div_iso} even without the generality assumption on $X,X'$ in all cases except quartic surfaces in $\bP^3$.
For completeness, we include a short argument which also covers quartic surfaces.

\begin{lemma} \label{lem:ram_div_iso}
    Let $X,X'$ be weighted double solids, with $X$ very general. If there is an isomorphism $Z \simeq Z'$, then $X \simeq X'$. 
\end{lemma}

\begin{proof} 
Denote the isomorphism of branch divisors $\phi \colon Z \simeq Z'$. Since $X$ is very general, $Z$ is Picard rank $1$, and hence so is $Z'$. This means that $\phi^*(h')=h$ where $h,h'$ are the ample generators of $\Pic(Z),\Pic(Z')$, respectively. Therefore, global sections of (multiples of) $h$ and $h'$ are identified, and hence
the embeddings of $Z,Z'$ into $\bP^3$ or $\bP(w_0, \dots, w_3)$ coincide up to a projective transformation. Hence, the covers $X \simeq X'$ are isomorphic.
\end{proof}

In the rest of the paper, we will often assume that weighted double solids satisfy $0<M\leq d$. This wil imply that the results of Section \ref{sec:covers} hold, including the limit case $\sA_Z \simeq \Db(Z)$.

\subsection{Hodge theory and Kuznetsov components}
Here we use the results of Section \ref{ssec:equivar_equivalences} to relate equivalences of Kuznetsov components to the cohomology of the ramification divisors.
Unless stated otherwise, we fix a weighted double solid $X$ with $0<M\leq d$. 
By Theorem \ref{thm:main}, we have the following semiorthogonal decomposition of the $\mu_n$-equivariant Kuznetsov component.
\begin{equation*}
    \sA_X^{\mu_2} = \langle \bL_{\sB} j_{0*} \Db(Z) , \cE_1, \dots, \cE_{d-M} \rangle \subset \Db(X)^{\mu_2},
\end{equation*}
where the $\cE_i$ are exceptional objects, so that $\langle \cE_i \rangle\simeq \Db(\pt)$.
By the discussion in Section \ref{ssec:TopKthy}, we have an isomorphism of Hodge structures 
\begin{equation}
    \label{eq:Ktop(AXm)}
    \rK_{-t}^\top(\sA_X^{\mu_2}) \simeq \rK_{-t}^\top(\Db(Z)) \oplus \bigoplus_{i=1}^{d-M}\rK_{-t}^\top(\Db(\pt))
\end{equation}
for all integers $t$.

If $\sC$ is a proper category, we will write 
\[\rK_0(\sC)^\perp \coloneqq \{ e \in \rK_0^\top(\sC) \mid \chi^{\mathrm{top}}(a,e)=0 \mbox{ for all } a\in \rK_0(\sC) \}.\]
Recall that the Chern character induces a functorial isomorphism of Hodge structures $\rK_0^\top(\Db(Z))_\bQ\simeq H^{\mathrm{even}}(Z,\bQ)$, by Proposition \ref{prop:HS_C_in_Stack}. 
We denote its image in $H^{\mathrm{even}}(Z,\bQ)$ by $\Lambda_Z$. Observe that, since $Z$ is a smooth surface, $\Lambda_Z \subset H^0(Z,\bZ)\oplus H^2(Z,\bZ) \oplus \frac 12H^4(Z,\bZ)$.

\begin{lemma}\label{lem:orthog_to_equiv_Kuznetsov_component_contain_in_primitive}
    Suppose that $X$ is a prime weighted double solid with $0<M$. Then
    \begin{enumerate}
        \item there is an isometry $\rK_0(\sA_X^{\mu_2})^\bot\simeq \rK_0(\Db(Z))^\bot$.
        \item The Chern character map in Proposition \ref{prop:HS_C_in_Stack} restricts to an isometry on $\rK_0(\Db(Z))^\perp$, and its image is contained in $H^2_{\prim}(Z,\bZ)\subset\Lambda_Z$.
    \end{enumerate}
\end{lemma}

\begin{proof}
    For (1), first suppose $0<M\leq d$. Then, by definition of $\rK_0(\sA_X^{\mu_2})^\bot$ and through the decomposition \eqref{eq:Ktop(AXm)}, a class $e\in \rK_0^\top(\sA_X^{\mu_2})$ lies in   $\rK_0(\sA_X^{\mu_2})^\bot$ if and only if $e\in \rK_0(\Db(Z))^\bot$ and $\chi^\top([\cE_i],e)=0$ for all $i$ (here we use that $[\cE_i]\in \rK_0(\sA_X^{\mu_2})$). Since \eqref{eq:Ktop(AXm)} is semiorthogonal, this is equivalent to saying that $e\in \rK_0(\Db(Z))^\bot$. Arguing similarly, if $d<M$ the decomposition \eqref{eq:KP_Z_sod} implies (1).

    To prove (2), we suppose $[E]\in\rK_0(\Db(Z))^\perp$. Let $h$ be an ample divisor on $Z$ (corresponding to a hyperplane section $H$), and fix a point $z\in Z$. Then $[\sO_Z],[\sO_h],[\sO_z]\in\rK_0(\Db(Z))$, so $[E]$ is orthogonal to them.
    The topological Euler pairing on $\Db(Z)$ satisfies Hirzebruch--Riemann--Roch.
    Since $\dim Z=2$, these orthogonality conditions become:
    \begin{align}
    0=\chi^\top(\sO_z,E)&=\mathrm{ch}_0(E),\label{eq: Euler pairing point and E}\\
    0=\chi^\top(\sO_h,E)&=\int_Z \left(0,-h,-\frac{h^2}{2}\right)\cdot\mathrm{ch}(E)\cdot\mathrm{td}(Z)=-h\cdot\mathrm{ch}_1(E) \label{eq: Euler pairing h and E}\\
    0= \chi_\top(\sO_Z,E)&=\int_Z \ch(E)\cdot\td(Z)=\ch_2(E) - \ch_1(E)\cdot \td_1(Z). \label{eq: Euler pairing Z and E}
    \end{align}
    Moreover, $\td_1(Z)$ is an algebraic class, i.e. $\td_1(Z)=\sum_i a_i E_i$ where for each $i$, $a_i\in\bQ$ and $E_i$ is an effective divisor. Again, since $E\in\rK_0(\sA_X^{\mu_2})^\perp$, we have:
    \begin{equation}
        0 = \chi^\top(\sO_{E_i},E)=\int_Z \left(0,-E_i,-\frac{E_i^2}{2}\right)\cdot\mathrm{ch}(E)\cdot\mathrm{td}(Z)=-E_i\cdot\mathrm{ch}_1(E).
    \end{equation}
    Then by \eqref{eq: Euler pairing Z and E}, $\ch_2(E)=0$. Together with \eqref{eq: Euler pairing point and E} and \eqref{eq: Euler pairing h and E}, it follows that $\ch(E)\in H^2_\prim(Z,\bZ)$.

    Finally, note that under the Chern character map, the Euler pairing $\chi^\top(-,-)$ and the intersection pairing on $H^*(Z,\bQ)$ coincide for classes of the form $(0,\ch_1,0)$. Therefore, $\ch$ restricts to an isometry on $\rK_0(\Db(Z))^\perp$.
\end{proof}

\begin{proposition} \label{prop:hs_pres_other}
Suppose that $X$ is a very general prime weighted double solid with $0<M$. Then the Chern character induces an isometry \[\rK_0(\sA_X^{\mu_2})^\perp \simeq H_\prim^2(Z, \bZ).\] 
\end{proposition}

\begin{proof}

Suppose $[E]\in\rK_0^\top(\Db(Z))$ satisfies $\ch(E)\in H^2_\prim(Z,\bZ)$. By Lemma \ref{lem:orthog_to_equiv_Kuznetsov_component_contain_in_primitive}, it is enough to show that $[E]\in\rK_0(\Db(Z))^{\perp}$. From the definition of $H^2_\prim(Z,\bZ)$, we have
\begin{equation}
    0 = \ch_0(E) = h\cdot \ch_1(E) = \ch_2(E).
\end{equation}
 Since $Z$ has Picard rank 1, we know
\[\mathrm{td}(Z)=\left(1,-\frac a2 h, \mathrm{td}_2(Z)\right),\]
for some integer $a\in\bZ$. In particular, $\ch_1(E)\cdot\td_1(Z)=0$. As in the proof of Lemma \ref{lem:orthog_to_equiv_Kuznetsov_component_contain_in_primitive}, by Hirzebruch--Riemann--Roch it follows that
\begin{equation}\label{eq:}
    0 = \chi^\top(\sO_{z},E)=\chi^\top(\sO_{h},E)=\chi^\top(\sO_{Z},E),
\end{equation}
where $z\in Z$. By our assumptions on $X$, the Picard group $\Pic Z$ has rank 1 and is generated by $h$. Hence $\rK_0(\Db(Z))$ has rank 3, and is generated by $[\sO_Z],[\sO_h],[\sO_z]$. Therefore $\chi^\top(a,E)=0$ for all $a\in \rK_0(\Db(Z))$, so $[E]\in\rK_0(\Db(Z))^{\perp}$. \qedhere
\end{proof}

We now prove the main theorem of this section by combining the results above.

\begin{theorem} \label{thm:main_hodge}
    Let $X$ and $X'$ be prime weighted double solids covering the same weighted projective space, with $X$ very general and $d=d'$. Suppose we have a Fourier--Mukai type equivalence $\Phi^{\mu_2} \colon \sA_X^{\mu_2} \to \sA_{X'}^{\mu_2}$. Then this induces a Hodge isometry
    \begin{equation*}H^{2}_\prim(Z, \bZ) \simeq H^{2}_\prim(Z', \bZ).
    \end{equation*}
\end{theorem}

\begin{proof}
We first show that $\Phi^{\mu_2}$ induces the required isomorphism of abelian groups preserving the bilinear pairings.  
By the functorial diagram \eqref{eq:KtopKalgCD}, and since the Euler pairing is defined through a canonical evaluation map, we immediately get an isometry $\gamma \colon \rK_0(\sA_X^{\mu_2})^\perp \simeq \rK_0(\sA_{X'}^{\mu_2})^\perp$. This fits into the commutative square 
\begin{equation}\label{eq:centralsquare}
    \begin{tikzcd}
         \rK_0(\sA_X^{\mu_2})^\perp \dar \arrow[r,"\gamma"] & \rK_0(\sA_{X'}^{\mu_2})^\perp \dar \\
      \rK_0^\top(\sA_X^{\mu_2}) \arrow[r,"\delta"] & \rK_0^\top(\sA_{X'}^{\mu_2}) 
    \end{tikzcd},
\end{equation} 
where $\delta$ is the Hodge isometry of Proposition \ref{prop:k_theory_hs_preserved}, and the vertical maps are functorial inclusions.

Since $X$ is very general, by Proposition \ref{prop:hs_pres_other} there is an isometry $\alpha \colon \rK_0(\sA_X^{\mu_2})^\perp \simeq H_\prim^2(Z, \bZ)$. Let $\alpha'$ be the composition of the maps in Lemma \ref{lem:orthog_to_equiv_Kuznetsov_component_contain_in_primitive} applied to $X'$, i.e.
\begin{equation*}
    \alpha' \colon \rK_0(\sA_{X'}^{\mu_2})^\perp\xrightarrow{\sim} \rK_0(\Db(Z))^\perp \xhookrightarrow{\ch} H^2_\prim(Z',\bZ).
\end{equation*}
This is an isometry onto its image. We claim that $\alpha'$ is in fact surjective. Consider the composition
\begin{equation}\label{eq:isometry_of_Hodge_structures}
    \alpha'\circ \gamma \circ \alpha^{-1} \colon H_\prim^2(Z,\bZ) \xlongrightarrow{\sim} \rK_0(\sA_X^{\mu_2})^\perp \xlongrightarrow{\sim}  \rK_0(\sA_{X'}^{\mu_2})^\perp \hooklongrightarrow H_\prim^2(Z',\bZ).
\end{equation}
Since $Z$ and $Z'$ are in the same linear series $|\sO_Y(nd)|$, we have $b_2(Z)=b_2(Z')$. Hence $\rk(H_\prim^2(Z, \bZ)) = b_2(Z)-1 = \rk(H_\prim^2(Z', \bZ))$. Therefore $\alpha'$ is an isometric isomorphism, and so is \eqref{eq:isometry_of_Hodge_structures}.

There is another commutative square 
\begin{equation}\label{eq:sidesquare}
    \begin{tikzcd}
        H_\prim^2(Z,\bZ) \arrow[d]  \arrow[r,"\alpha^{-1}"] & \rK_0(\sA_X^{\mu_2})^\perp \dar  \\
      \Lambda_Z \rar  & \rK_0^\top(\sA_X^{\mu_2}) 
    \end{tikzcd},
\end{equation} 
where the vertical maps are the natural inclusions, and the bottom horizontal arrow is the inverse of the Chern character $\Lambda_Z \to \rK_0^\top(Z)$ followed by the inclusion of the first factor of \eqref{eq:Ktop(AXm)}. The commutativity of the diagram implies that   $\rK_0(\sA_X^{\mu_2})^\perp$ carries a Hodge structure which, on the one hand, makes $\alpha$ into a morphism of Hodge structures, and on the other hand is induced by restricting the Hodge structure of  $\rK^\top (\sA_X^{\mu_2})$.
The same arguments show that $\alpha'$ is a morphism of Hodge structures.
In particular, \eqref{eq:centralsquare} is a commutative diagram of Hodge structures, where the horizontal arrows are isometries.

Therefore \eqref{eq:isometry_of_Hodge_structures} is the desired isometry of Hodge structures $H^{2}_\prim(Z, \bZ) \simeq H^{2}_\prim(Z', \bZ)$.
\end{proof}

\begin{remark}\label{rmk:X4Hodge}
Alternatively, if $\rK_0(\sC)^\perp$ vanishes for $\sC={}^\perp\sA_X^{\mu_2}$ and $\sC= j_{0*}\Db(Z)^\perp$ (with orthogonals taken in $\Db(X)^{\mu_2}$), one can argue as above that $\rK_0(\sA_X^{\mu_2})^\bot\simeq \rK_0(\Db(Z))^\bot$ by comparing the decompositions 
    \begin{equation*}
        \langle \sA_X^{\mu_2}, {}^\perp\sA_X^{\mu_2}\rangle = \Db(X)^{\mu_2}= \langle j_{0 *} \Db(Z)^\perp, j_{0*} \Db(Z) \rangle.
    \end{equation*}
    given in Theorem \ref{prop:sod_2} and Proposition \ref{prop:sod_1}. For example, one shows in this way that Proposition \ref{prop:hs_pres_other} and Theorem \ref{thm:main_hodge} hold for very general members of the family of Fano threefolds $X_4$ because of the decompositions \eqref{eq:sod_X4}.
\end{remark}

\begin{remark}\label{remark: generalisation of main_hodge to arbitrary dimension}
It would be interesting to know whether the statement of Proposition \ref{prop:hs_pres_other} extends to higher dimensions, since then the argument of Theorem \ref{thm:main_hodge} would also carry through.  We hope to explore this further in future work.
\end{remark}

\section{Categorical Torelli theorems}
\label{sec:torelli}

In this section, we apply Theorem \ref{thm:main_hodge} to prove a categorical Torelli theorem for the family of Fano threefolds $X_2$ (Theorem  \ref{thm:X_2}) and for the family of Fano threefolds $Y_1$ (Theorem \ref{thm:Y_1}). The latter was the only remaining index 2 open case.

Consider once again a weighted double solid $X$ satisfying $0<M$. Denote by $\tau \colon \sA_X \to \sA_X$ the categorical involution induced by the involution of the double cover $X$. Also define the \emph{rotation functor} $\sfR \colon \sA_X \to \sA_X$ by $\sfR(-) \coloneqq \bL_{\sB_X}(- \otimes \sO_X(1))$. 
We have the following relationships between the Serre functor of $\sA_X$, the rotation functor $\sfR$, and the categorical involution $\tau$:

\begin{proposition} \leavevmode \label{prop:serre_rotation_tau_relations}
There are isomorphisms of functors:
    \begin{itemize} 
        \item $\sfR^d \simeq \tau[1]$,
        \item $\sfS^{-1}_{\sA_X} \simeq \sfR^{M}[-m+1]$.
    \end{itemize}
\end{proposition}

\begin{proof}
    The first bullet point is by \cite[Corollary 3.18]{kuznetsov2019calabi}.
    For the second bullet point, note that by Proposition \ref{prop:serre_subcategory} and Remark \ref{rem:X_is_Fano} we get
    \begin{align*}
    \sfS^{-1}_{\sA_X}(-) &= \bL_{\sB} \bL_{\sB(1)} \cdots \bL_{\sB(M-1)}(- \otimes \sO_X(M) )[-m+1] \\
    &= \sfR \bL_{\sB} \bL_{\sB(1)} \cdots \bL_{\sB(M-2)}(- \otimes \sO_X(M-1) )[-m+1] \\
    &\cdots \\
    &=\sfR^{M}[-m+1],
    \end{align*}
    by repeatedly using that $\bL_{F}(-\otimes \sO_X(1))=\bL_{F\otimes \sO_X(-1)}(-)\otimes \sO_X(1)$.
\end{proof}

Let $X'$ be another weighted double solid with the same values  of $m,d$ as $X$. We denote by $\tau \in \Aut(\sA_X)$ and $\tau' \in \Aut(\sA_{X'})$ the covering autoequivalences.

\begin{lemma} \label{lem:eq_equivalence_descends}
Suppose that an equivalence $\Phi \colon \sA_{X} \simeq \sA_{X'}$ commutes with the involutions, i.e. there is an isomorphism of functors $\Phi \circ \tau \simeq \tau'\circ \Phi$. Then $\Phi$ descends to an equivalence of equivariant categories $\Phi^{\mu_2} \colon \sA_{X}^{\mu_2} \simeq \sA_{X'}^{\mu_2}$.
\end{lemma}

\begin{proof}
    All actions we discuss in this proof will be understood to be $\bZ/2 \simeq \mu_2$-actions. We first observe that $\Phi$ preserves $1$-categorical actions (see Definition \ref{def:categorical_action}(1)). This is equivalent to the fact that $\Phi$ commutes with the involutions.

    Next, we check that $\Phi$ intertwines the involutions, considered as $2$-categorical actions on $\sA_X$ and $\sA_{X'}$ (Definition \ref{def:categorical_action}(2)). The $1$-categorical actions $\tau$ and $\tau'$ lift to $2$-categorical actions because the functors $\tau \colon \sA_X \to \sA_X$ and $\tau' \colon \sA_{X'} \to \sA_{X'}$ are given by pulling back geometric involutions and pullbacks are functorial.
    Moreover, these lifts are unique because $H^2(B \bZ/2, \bC^\times)= 0 $ (see \cite[Corollary 3.4]{bayer2023kuznetsov} for the lifting criterion, and \cite[Example 3.12]{bayer2023kuznetsov} for the vanishing). Thus $\Phi$ sends the $2$-categorical action $\tau$ to the unique $2$-categorical action $\tau'$. So $\Phi$ respects $2$-categorical actions, as required. 
\end{proof}

\begin{lemma}\label{lem:Mdivides}
    Assume that $M$ divides an odd multiple of $d$. Then an equivalence $\Phi \colon \sA_X \simeq \sA_{X'}$ commutes with the covering involutions.
\end{lemma}

\begin{proof}
Write $d(2b+1)=Ma$ for some integers $a,b$. Using Proposition \ref{prop:serre_rotation_tau_relations}, up to shifts we have 
\begin{equation*}
        \tau = \tau^{2b+1} \simeq \sfR^{d(2b+1)}
        = \sfR^{M a}
        \simeq \sfS_{\sA_X}^{-a},
    \end{equation*}
and $\tau'\simeq \sfS_{\sA_{X'}}^{-a}$ up to the same shift.
Then, because Serre functors and shifts commute with equivalences of categories, $\Phi$ commutes with the involution. 
\end{proof}

\subsection{The case of \texorpdfstring{$X_2$}{X2}}

We now consider the case where $Y=\bP^3$. In this case, the constraint $0<M=4-d\leq d$ leaves us with two families: sextic double solids for $d=3$ (see Example \ref{ex:sods}\eqref{itm:X2}) and quartic double solids for $d=2$ (see Section \ref{sec:Y_2}). We first prove a categorical Torelli theorem for sextic double solids $X_2$.

\begin{theorem} \label{thm:X_2}
    Let $X, X'$ be prime Fano threefolds of index $1$, and genus $2$, with $X$ very general. Then an equivalence $\sA_X \simeq \sA_{X'}$ of Fourier--Mukai type implies $X \simeq X'$.
\end{theorem}

\begin{proof}
In this case, $M=m-d=1$ always divides $d$, so the equivalence $\Phi$ descends to the equivariant categories by \ref{lem:eq_equivalence_descends} and \ref{lem:Mdivides}.

Then, by Theorem \ref{thm:main_hodge} we get a Hodge isometry
    \begin{equation*}
        H^{ 2}_\prim(Z, \bQ) \simeq H^{2}_\prim(Z', \bQ).
    \end{equation*}
    Next, we apply the Torelli theorem for generic hypersurfaces \cite[p. 325]{donagi1983generic} (see \cite[Theorem 0.2]{voisin2022schiffer} for the statement in terms of a Hodge isometry), which implies $Z\simeq Z'$. Donagi's theorem applies to the present case, since its numerical assumptions are automatically satisfied as long as $d\neq 1,2$.
    Now, by Lemma \ref{lem:ram_div_iso}, we conclude that $X \simeq X'$.
\end{proof}

\subsection{The case of \texorpdfstring{$Y_1$}{Y1}}

 We now consider a double cover $X$ of $\bP(1,1,1,2)$ branched in a sextic hypersurface $Z$. This cover is known as a \emph{Veronese double cone} and is often denoted $Y_1$ in the literature (see Example \ref{ex:sods}\eqref{itm:veronese}). These are prime Fano threefolds of index $2$, and degree $1$.

\begin{theorem} \label{thm:Y_1}
    Let $X,X'$ be Veronese double cones, with $X$ very general. Suppose that we have an equivalence $\Phi\colon\sA_X \xrightarrow{\sim} \sA_{X'}$ of Fourier--Mukai type which commutes with the covering involution.
    Then $X \simeq X'$.
\end{theorem}

\begin{proof}
    By assumption and Lemma \ref{lem:eq_equivalence_descends}, $\Phi$ descends to an equivariant equivalence $\sA_{X}^{\mu_2} \simeq \sA_{X'}^{\mu_2}$, Then, by Theorem \ref{thm:main_hodge} we have a Hodge isometry $H^{ 2}_\prim(Z, \bQ) \simeq H^{2}_\prim(Z', \bQ)$. Now by \cite[Theorem A]{saito1986weak}, generic Torelli holds for the degree $6$ hypersurface $Z \subset \bP(1,1,1,2)$. Thus $Z \simeq Z'$ and by Lemma \ref{lem:ram_div_iso} we get $X \simeq X'$.
\end{proof}

\begin{remark}\label{rmk:torelliY1}
    The relations in Proposition \ref{prop:serre_rotation_tau_relations} read $\sfR^3 \simeq \tau[1]$ and $\sfS_{\sA_X}^{-1}\simeq \sfR^2[-3]$. Since any equivalence must commute with shifts and Serre functors, then $\Phi$ commutes with $\tau$ if and only if it commutes with $\sfR$.
\end{remark}

\subsection{The case of \texorpdfstring{$Y_2$}{Y2}} \label{sec:Y_2}

We now consider a double cover $X$ of $\bP^3$ branched in a quartic K3 $Z \subset \bP^3$. This cover is known as the \emph{quartic double solid} and is often referred to as $Y_2$ in the literature. Below, we give a new proof of a categorical Torelli theorem for quartic double solids. Other proofs have been obtained in \cite{bernardara2016semi, APR19, bayer2023kuznetsov, feyzbakhsh2023new}.

\begin{theorem} \label{thm:Y_2}
    Let $X, X'$ be quartic double solids, with $X$ very general. Then an equivalence $\sA_X \simeq \sA_{X'}$ implies $X \simeq X'$.
\end{theorem}

\begin{proof}
    Firstly, note that all equivalences in the quartic double solid case are of Fourier--Mukai type by \cite[Theorem 1.3]{quarticdoublesolidhearts}. Also note that Lemma \ref{lem:Mdivides} holds in this case. Thus, as in the proof of Theorem \ref{thm:X_2}, the equivalence implies a Hodge isometry $H_\prim^2(Z, \bQ) \simeq H_\prim^2(Z', \bQ)$. Since $X$ is very general, $Z$ has Picard rank 1. This means that that the primitive second cohomology of $Z$ is Hodge isometric to the transcendental lattice of $Z$ \cite[Lemma 3.3.1]{huybrechts2016lectures}. So the Hodge isometry on primitive cohomology gives an isometry between the transcendental lattices of $Z$ and $Z'$. Therefore, by \cite[Corollary 16.3.7]{huybrechts2016lectures}, $Z$ and $Z'$ are derived equivalent via a Fourier--Mukai functor. By \cite[Theorem 1.7]{oguiso2002k3}, $Z$ can only have one trivial Fourier--Mukai partner, so $Z \simeq Z'$. Finally, by Lemma \ref{lem:ram_div_iso} we get $X \simeq X'$. 
\end{proof}


\providecommand{\bysame}{\leavevmode\hbox to3em{\hrulefill}\thinspace}
\providecommand{\MR}{\relax\ifhmode\unskip\space\fi MR }
\providecommand{\MRhref}[2]{%
  \href{http://www.ams.org/mathscinet-getitem?mr=#1}{#2}
}


\bibliographystyle{alpha}
{\small{\bibliography{mybib2}}}

\newcommand{\etalchar}[1]{$^{#1}$}
\begin{thebibliography}{BMMS12}

\bibitem[APR22]{APR19}
Matteo Altavilla, Marin Petkovi\'{c}, and Franco Rota.
\newblock Moduli spaces on the {K}uznetsov component of {F}ano threefolds of
  index 2.
\newblock {\em \'{E}pijournal G\'{e}om. Alg\'{e}brique}, 6:Art. 13, 31, 2022.

\bibitem[BBF{\etalchar{+}}22]{bayer2022desingularization}
Arend Bayer, Sjoerd Beentjes, Soheyla Feyzbakhsh, Georg Hein, Diletta
  Martinelli, Fatemeh Rezaee, and Benjamin Schmidt.
\newblock The desingularization of the theta divisor of a cubic threefold as a
  moduli space.
\newblock {\em Geom. Topol.}, 2022.

\bibitem[BFK14]{BFK}
Matthew Ballard, David Favero, and Ludmil Katzarkov.
\newblock A category of kernels for equivariant factorizations and its
  implications for {H}odge theory.
\newblock {\em Publ. Math. Inst. Hautes \'{E}tudes Sci.}, 120:1--111, 2014.

\bibitem[Bla16]{blanc2016topological}
Anthony Blanc.
\newblock Topological {K}-theory of complex noncommutative spaces.
\newblock {\em Compos. Math.}, 152(3):489--555, 2016.

\bibitem[BLMS23]{bayer2017stability}
Arend Bayer, Mart\'{\i} Lahoz, Emanuele Macr\`{i}, and Paolo Stellari.
\newblock Stability conditions on {K}uznetsov components.
\newblock {\em Ann. Sci. \'{E}c. Norm. Sup\'{e}r. (4)}, 56(2):517--570, 2023.
\newblock With an appendix by Bayer, Lahoz, Macr\`{i}, Stellari and X. Zhao.

\bibitem[BMMS12]{bernardara2012categorical}
Marcello Bernardara, Emanuele Macr\`{i}, Sukhendu Mehrotra, and Paolo Stellari.
\newblock A categorical invariant for cubic threefolds.
\newblock {\em Adv. Math.}, 229(2):770--803, 2012.

\bibitem[BO01]{bondal2001reconstruction}
Alexey Bondal and Dmitri Orlov.
\newblock Reconstruction of a variety from the derived category and groups of
  autoequivalences.
\newblock {\em Compos. Math.}, 125(3):327--344, 2001.

\bibitem[BP23]{bayer2023kuznetsov}
Arend Bayer and Alexander Perry.
\newblock Kuznetsov's {F}ano threefold conjecture via {K}3 categories and
  enhanced group actions.
\newblock {\em J. Reine Angew. Math.}, 800:107--153, 2023.

\bibitem[BT16]{bernardara2016semi}
Marcello Bernardara and Gon\c{c}alo Tabuada.
\newblock From semi-orthogonal decompositions to polarized intermediate
  {J}acobians via {J}acobians of noncommutative motives.
\newblock {\em Moscow Math. J.}, 16(2):205--235, 2016.

\bibitem[Del97]{deligne1997action}
Pierre Deligne.
\newblock Action du groupe des tresses sur une cat\'{e}gorie.
\newblock {\em Invent. Math.}, 128(1):159--175, 1997.

\bibitem[Don83]{donagi1983generic}
Ron Donagi.
\newblock Generic {T}orelli for projective hypersurfaces.
\newblock {\em Compos. Math.}, 50(2-3):325--353, 1983.

\bibitem[Ela12]{elagin2012descent}
Alexey Elagin.
\newblock Descent theory for semiorthogonal decompositions.
\newblock {\em Sbornik: Math.}, 203(5):645, 2012.

\bibitem[Ela14]{elagin2014equivariant}
Alexey Elagin.
\newblock On equivariant triangulated categories.
\newblock {\em arXiv:1403.7027}, 2014.

\bibitem[FL23]{FL23}
Yu-Wei Fan and Kuan-Wen Lai.
\newblock {Fourier--Mukai numbers of K3 categories of very general special
  cubic fourfolds}.
\newblock {\em arXiv:2307.14486}, 2023.

\bibitem[FLZ23]{feyzbakhsh2023new}
Soheyla Feyzbakhsh, Zhiyu Liu, and Shizhuo Zhang.
\newblock {New perspectives on categorical Torelli theorems for del Pezzo
  threefolds}.
\newblock {\em arXiv:2304.01321}, 2023.

\bibitem[HLP20]{DHLncHDR}
Daniel Halpern-Leistner and Daniel Pomerleano.
\newblock Equivariant {H}odge theory and noncommutative geometry.
\newblock {\em Geom. Topol.}, 24(5):2361--2433, 2020.

\bibitem[HO23]{hirano2023derived}
Yuki Hirano and Genki Ouchi.
\newblock Derived factorization categories of non-{T}hom-{S}ebastiani-type sums
  of potentials.
\newblock {\em Proc. London Math. Soc. (3)}, 126(1):1--75, 2023.

\bibitem[HR19]{huybrechts2016hochschild}
Daniel Huybrechts and J\o rgen~Vold Rennemo.
\newblock Hochschild cohomology versus the {J}acobian ring and the {T}orelli
  theorem for cubic fourfolds.
\newblock {\em Algebraic Geom.}, 6(1):76--99, 2019.

\bibitem[Huy06]{huybrechts2006fourier}
Daniel Huybrechts.
\newblock {\em Fourier--{M}ukai transforms in algebraic geometry}.
\newblock Oxford Math. Monogr. The Clarendon Press, Oxford University Press,
  Oxford, 2006.

\bibitem[Huy16]{huybrechts2016lectures}
Daniel Huybrechts.
\newblock {\em Lectures on {K}3 surfaces}, volume 158 of {\em Cambridge Stud.
  Adv. Math.}
\newblock Cambridge University Press, Cambridge, 2016.

\bibitem[Huy17]{Huy17_K3CatCubic}
Daniel Huybrechts.
\newblock The {K}3 category of a cubic fourfold.
\newblock {\em Compos. Math.}, 153(3):586--620, 2017.

\bibitem[JLLZ21]{jacovskis2021categorical}
Augustinas Jacovskis, Xun Lin, Zhiyu Liu, and Shizhuo Zhang.
\newblock {Categorical Torelli theorems for Gushel--Mukai threefolds}.
\newblock {\em arXiv:2108.02946}, 2021.

\bibitem[JLZ22]{jacovskis2022brill}
Augustinas Jacovskis, Zhiyu Liu, and Shizhuo Zhang.
\newblock {Brill--Noether theory for Kuznetsov components and refined
  categorical Torelli theorems for index one Fano threefolds}.
\newblock {\em arXiv:2207.01021}, 2022.

\bibitem[Kap88]{Kapranov}
Mikhail~M. Kapranov.
\newblock On the derived categories of coherent sheaves on some homogeneous
  spaces.
\newblock {\em Invent. Math.}, 92(3):479--508, 1988.

\bibitem[Kol19]{KollarHyper}
J\'{a}nos Koll\'{a}r.
\newblock Algebraic hypersurfaces.
\newblock {\em Bull. Amer. Math. Soc. (N.S.)}, 56(4):543--568, 2019.

\bibitem[KP17]{kuznetsov2017derived}
Alexander Kuznetsov and Alexander Perry.
\newblock Derived categories of cyclic covers and their branch divisors.
\newblock {\em Selecta Math. (N.S.)}, 23(1):389--423, 2017.

\bibitem[KP21]{kuznetsov2021serre}
Alexander Kuznetsov and Alexander Perry.
\newblock {Serre functors and dimensions of residual categories}.
\newblock {\em arXiv:2109.02026}, 2021.

\bibitem[K{\"u}n23]{Kung22}
Felix K{\"u}ng.
\newblock {Twisted Hodge Diamonds give rise to non-Fourier--Mukai functors}.
\newblock {\em J. Noncommutative Geom.}, 2023.
\newblock To appear.

\bibitem[Kuz09]{kuznetsov2009derived}
Alexander Kuznetsov.
\newblock Derived categories of {F}ano threefolds.
\newblock {\em Tr. Mat. Inst. Steklova}, 264:116--128, 2009.

\bibitem[Kuz10]{kuznetsov2010cubic}
Alexander Kuznetsov.
\newblock Derived categories of cubic fourfolds.
\newblock In {\em Cohomological and geometric approaches to rationality
  problems}, volume 282 of {\em Progr. Math.}, pages 219--243. Birkh\"{a}user
  Boston, Boston, MA, 2010.

\bibitem[Kuz19]{kuznetsov2019calabi}
Alexander Kuznetsov.
\newblock Calabi--{Y}au and fractional {C}alabi--{Y}au categories.
\newblock {\em J. Reine Angew. Math.}, 753:239--267, 2019.

\bibitem[LPS23]{LPS23}
Mart\'i Lahoz, Laura Pertusi, and Paolo Stellari.
\newblock Categorical {T}orelli theorems for weighted hypersurfaces.
\newblock {\em upcoming}, 2023.

\bibitem[LPZ23]{LPZ18twistedcubics}
Chunyi Li, Laura Pertusi, and Xiaolei Zhao.
\newblock Twisted cubics on cubic fourfolds and stability conditions.
\newblock {\em Algebr. Geom.}, 10(5):620--642, 2023.

\bibitem[LZ23]{lin2023serre}
Xun Lin and Shizhuo Zhang.
\newblock {Serre algebra, matrix factorization and categorical Torelli theorem
  for hypersurfaces}.
\newblock {\em arXiv:2310.09927}, 2023.

\bibitem[MS19]{MSK3Lectures}
Emanuele Macr\`{i} and Paolo Stellari.
\newblock Lectures on non-commutative {K}3 surfaces, {B}ridgeland stability,
  and moduli spaces.
\newblock In {\em Birational geometry of hypersurfaces}, volume~26 of {\em
  Lect. Notes Unione Mat. Ital.}, pages 199--265. Springer, Cham, [2019]
  \copyright 2019.

\bibitem[Ogu02]{oguiso2002k3}
Keiji Oguiso.
\newblock K3 surfaces via almost-primes.
\newblock {\em Math. Res. Lett.}, 9(1):47--63, 2002.

\bibitem[Orl09]{Orlov09_DerSing}
Dmitri Orlov.
\newblock Derived categories of coherent sheaves and triangulated categories of
  singularities.
\newblock In {\em Algebra, arithmetic, and geometry: in honor of {Y}u. {I}.
  {M}anin. {V}ol. {II}}, volume 270 of {\em Progr. Math.}, pages 503--531.
  Birkh\"{a}user Boston, Boston, MA, 2009.

\bibitem[Per21]{pertusi2021fourier}
Laura Pertusi.
\newblock Fourier-{M}ukai partners for very general special cubic fourfolds.
\newblock {\em Math. Res. Lett.}, 28(1):213--243, 2021.

\bibitem[Per22]{perry2022integral}
Alexander Perry.
\newblock The integral {H}odge conjecture for two-dimensional {C}alabi-{Y}au
  categories.
\newblock {\em Compos. Math.}, 158(2):287--333, 2022.

\bibitem[Pir22]{pirozhkov2022categorical}
Dmitrii Pirozhkov.
\newblock Categorical {T}orelli theorem for hypersurfaces.
\newblock {\em arXiv:2208.13604}, 2022.

\bibitem[PLZ23]{quarticdoublesolidhearts}
Laura Pertusi, Chunyi Li, and Xiaolei Zhao.
\newblock {Derived categories of hearts on Kuznetsov components}.
\newblock {\em J. London Math. Soc.}, 108(2):2146–--2174, 2023.

\bibitem[PS23]{PSSurvey}
Laura Pertusi and Paolo Stellari.
\newblock Categorical {T}orelli theorems: results and open problems.
\newblock {\em Rend. Circ. Mat. Palermo (2)}, 72(5):2949--3011, 2023.

\bibitem[PY22]{pertusi2022some}
Laura Pertusi and Song Yang.
\newblock Some remarks on {F}ano threefolds of index two and stability
  conditions.
\newblock {\em Int. Math. Res. Not. IMRN}, (17):12940--12983, 2022.

\bibitem[Sai86]{saito1986weak}
Masa-Hiko Saito.
\newblock Weak global {T}orelli theorem for certain weighted projective
  hypersurfaces.
\newblock {\em Duke Math. J.}, 53(1):67--111, 1986.

\bibitem[Ste03]{Ste03}
Jan Stevens.
\newblock {\em Deformations of singularities}, volume 1811 of {\em Lect. Notes
  Math.}
\newblock Springer-Verlag, Berlin, 2003.

\bibitem[Tab08]{tabuada}
Gon\c{c}alo Tabuada.
\newblock Higher {$K$}-theory via universal invariants.
\newblock {\em Duke Math. J.}, 145(1):121--206, 2008.

\bibitem[To{\"{e}}07]{toen2007homotopy}
Bertrand To{\"{e}}n.
\newblock The homotopy theory of {$dg$}-categories and derived {M}orita theory.
\newblock {\em Invent. Math.}, 167(3):615--667, 2007.

\bibitem[Voi07]{Voisin}
Claire Voisin.
\newblock {\em Hodge theory and complex algebraic geometry. {I}}, volume~76 of
  {\em Cambridge Stud. Adv. Math.}
\newblock Cambridge University Press, Cambridge, english edition, 2007.
\newblock Translated from the French by Leila Schneps.

\bibitem[Voi22]{voisin2022schiffer}
Claire Voisin.
\newblock Schiffer variations and the generic {T}orelli theorem for
  hypersurfaces.
\newblock {\em Compos. Math.}, 158(1):89--122, 2022.

\bibitem[Wei97]{weibel}
Charles Weibel.
\newblock The {H}odge filtration and cyclic homology.
\newblock {\em $K$-Theory}, 12(2):145--164, 1997.

\end{thebibliography}

\end{document}